\documentclass[a4paper]{amsart}
\usepackage[latin1]{inputenc}
\usepackage[T1]{fontenc}
\usepackage[french, english]{babel}
\usepackage{graphicx}
\usepackage{endnotes}
\usepackage{amsthm, amsmath}
\usepackage{amssymb}
\usepackage{xcolor}
\usepackage{graphics,graphicx}
\usepackage{tikz}
\usetikzlibrary{fit,calc,positioning,decorations.pathreplacing,matrix}
\usetikzlibrary{trees}
\usepackage[dvips,breaklinks]{hyperref}
\usepackage{textcomp}
\usepackage{stmaryrd}
\usepackage{enumerate}
\hypersetup{
pdfpagemode=UseOutlines, 
pdfstartview=Fit, 
pdffitwindow=true, 
pdfpagelayout=TwoColumnsRight,
pdftoolbar=true, 
pdfmenubar=true, 
bookmarksopen=false, 
bookmarksnumbered=true, 
colorlinks=true, 
linkcolor=blue, 
urlcolor=blue, 
anchorcolor=black, 
citecolor=green, 
frenchlinks=true, 
pdfborder={0 0 0} 
}

\numberwithin{equation}{section} 
\numberwithin{figure}{section} 
 \let\footnote=\endnote
\ifx\thesection\undefined
\newtheorem{thm}{Theorem}
\else
\newtheorem{thm}{Theorem}[section]
\fi
 \theoremstyle{definition}
 \newtheorem{defi}[thm]{Definition}
 \newtheorem{exple}[thm]{Example}

 \theoremstyle{plain}
 \newtheorem{prop}[thm]{Proposition}
 \newtheorem{lem}[thm]{Lemma}
 
 \newtheorem*{thmnonnum}{Theorem}
 \newtheorem*{propnonnum}{Proposition}

\newcommand{\CP}{C(\mathcal{P})}
\newcommand{\HP}{H(\mathcal{P})}
\newcommand{\HTn}{\widehat{HT_n}}
\newcommand{\HTi}{\widehat{HT_i}}

\newcommand{\HTjpu}{\widehat{HT_{j+1}}}
\newcommand{\HTd}{\widehat{HT_2}}
\newcommand{\HTt}{\widehat{HT_3}}
\newcommand{\HTq}{\widehat{HT_4}}
\newcommand{\HTc}{\widehat{HT_5}}
\newcommand{\htn}{HT_n}
\newcommand{\HHTn}{\mathcal{H}_{\widehat{HT}}}
\newcommand{\intunn}{\llbracket 1,n \rrbracket}
\newcommand{\Spnr}{\mathcal{B}_{HT}}

\newcommand{\hn}{h_n}
\newcommand{\Dr}{\Delta}
\newcommand{\dnap}{d^n_{\alpha,\pi}}
\newcommand{\cnap}{c^n_{\alpha,\pi}}
\newcommand{\cnapp}{c^{n \bullet}_{\alpha,\pi}}

\title{Hypertree posets and hooked partitions}

\author{B\'{e}r\'{e}nice Oger}

\thanks{ Institut Camille Jordan, UMR 5208, Universit\'{e} Claude Bernard Lyon 1 \\
 B\^{a}t. Jean Braconnier nb 101, 43 Bd du 11 novembre 1918, 69622 Villeurbanne Cedex \\
 \textbf{e-mail address:} oger@math. univ-lyon1. fr}

\begin{document}

\begin{abstract}
We adapt here the computation of characters on incidence Hopf algebras introduced by W. Schmitt in the 1990s to a family mixing bounded and unbounded posets. We then apply our results to the family of hypertree posets and partition posets. As a consequence, we obtain some enumerative formulas and a new proof for the computation of the Moebius numbers of the hypertree posets. Moreover, we compute the coproduct of the incidence Hopf algebra and recover a known formula for the number of hypertrees with fixed valency set and edge sizes set. 
\end{abstract}

\maketitle


\tableofcontents

\section*{Introduction}
In 1994, W. Schmitt defined in his article \cite{IHA} the notion of incidence Hopf algebra associated to a given family of posets satisfying some closure conditions. Using the structure of Hopf algebra, one can define a convolution on characters of this algebra. The Moebius number for posets of the family can then be computed using characters on the incidence Hopf algebra. 

However, the incidence Hopf algebras of W. Schmitt are only defined for bounded posets. We introduce in this article a way to compute some characters for another type of posets, called the triangle and diamond posets. The diamond posets are bounded posets whereas the triangle posets have a least element but no greatest one. If we consider the hereditary family generated by the diamond posets and the augmented triangle posets, i.e. the triangle posets with an added greatest element, we can build the associated incidence Hopf algebra $\mathcal{H}$. The coproduct in the bialgebra $\mathcal{B}$ generated by isomorphism classes in the hereditary family obtained from diamond and triangle posets can be linked with the coproduct of the incidence Hopf algebra $\mathcal{H}$: this relation enables us to identify a computation on maps from the bialgebra $\mathcal{B}$ to $\mathbb{Q}$ with the convolution of characters on the incidence Hopf algebra $\mathcal{H}$. The advantage of this method is that the computation of such maps on the bialgebra $\mathcal{B}$ is, in most cases, easier. In the rest of the article, we will apply this theorem to hypertree posets.

\medskip

In the third part of the article, we recall the notion of hypertrees. Hypergraphs have been introduced in $1989$ by C. Berge in \cite{Berge} as a generalization of graphs. Hypertrees are hypergraphs satisfying a kind of connectedness and acyclicity. The set of hypertrees on a vertex set $I$ can be endowed with a partial order given by union of edges. We prove the following criterion on $\alpha_i$ and $\pi_j$ for the existence of a hypertree with $\alpha_i$ vertices of valency $i$ and $\pi_j$ edges of size $j$:
\begin{equation}\label{crit}
\sum_{i=1}^k \alpha_i = n,\quad \sum_{j=2}^{l} (j-1) \pi_j = n-1 \quad \text{and} \quad \sum_{i=1}^k i \alpha_i = n + \sum_{j=2}^{l} \pi_j -1. 
\end{equation} 

Then, using this criterion, we compute the coproduct in the bialgebra $\mathcal{B}_{HT}$ associated with hypertree posets $h_n$ and partition posets $p_n$. The coproduct is given by the following formula:
\begin{thmnonnum} If the set $\mathcal{P}(n)$ is the set of tuples $\alpha=(\alpha_1, \dots, \alpha_k)$ and $\pi=(\pi_2, \dots, \pi_l)$ satisfying Equations \eqref{crit}, the coproduct of $h_n$ in $\mathcal{B}_{HT}$ is given by:
\begin{equation*}
\Delta(h_n)= \frac{1}{n} \times \sum_{(\alpha, \pi) \in \mathcal{P}(n)} \frac{n!}{\prod_{j \geq 2} (j-1)!^{\pi_j} \pi_j!} \times \frac{k! \times n!}{\prod_{i \geq 1} (i-1)!^{\alpha_i} \alpha_i!} \prod_{i=2}^k p_i^{\alpha_i} \otimes \prod_{j=2}^l h_j^{\pi_j},
\end{equation*}
with $k= \sum_{j\geq 2}\pi_j -1$. 
\end{thmnonnum}
This formula is linked with the number of hypertrees of fixed valencies and edge sizes sets, which was also computed by M. Bousquet-Mélou and G. Chapuy in \cite{BMC} (see also \cite{Stanley2001}), in terms of bicoloured trees, and by R. Bacher in \cite{BacHyp} in terms of hypertrees.
The first step of this computation is to show that it can be reduced to the computation of the number of ways to build a hypertree from a $\pi$-hooked partition. Then we show with a proof using a Pr\"{u}fer code that this construction is encoded by words. Finally, we count these words. 

\medskip

The Moebius number of the poset of hypertrees on $n$ vertices has been computed by J. McCammond and J. Meier in $2004$ in the article \cite{McCM}. F. Chapoton has computed its characteristic polynomials in \cite{ChHyp} and has conjectured the action of the symmetric group on the homology of the hypertree poset, which has been proven in \cite{mar1}. We give a new proof for the computation of the Moebius number of the hypertree posets at the end of the article. This computation gives the following enumerative formula:
\begin{propnonnum}The following equality holds:
\begin{equation*}
(n-1)^{n-2} = \sum_{(\alpha, \pi) \in \mathcal{P}(n)} \frac{(-1)^{i \alpha_i-1}}{n} \times \frac{n!}{\prod_{j \geq 2} (j-1)!^{\pi_j} \pi_j!} \times \frac{k! \times n!}{\prod_{i \geq 1} \alpha_i!},
\end{equation*}
where $\mathcal{P}(n)$ is the set of pairs $(\alpha_i,\pi_j)$ satisfying Equations \eqref{crit}.
\end{propnonnum}

\section{Generalities on posets and incidence Hopf algebras}

We introduce in this section some general notions on posets and incidence Hopf algebras which will be needed in this article. 

	\subsection{Generalities on posets}

A \emph{poset} is a set endowed with a partial order $\leq$. We called \emph{trivial} the poset which has only one element. If $P$ is a poset in which $x \leq y$, then the \emph{interval} $[x,y]$ is the set $\{z \in P: x \leq z \leq y\}$ and the \emph{half-open interval} $[x,y)$ is the set $\{z \in P: x \leq z < y\}$. If $P$ is an interval, it is said to be a \emph{bounded} poset. In this case, its least and greatest elements will be respectively denoted by $\hat{0}_P$ and $\hat{1}_P$, or $\hat{0}$ and $\hat{1}$ if there is no ambiguity. 

Let us define the following poset invariant: 

\begin{defi} The \emph{Moebius function} $\mu$ is recursively defined on a poset $P$ by \label{defmoeb}
\begin{eqnarray*}
\mu(x,x)=1, & \forall x \in P \\
\mu(x,y)=-\sum_{x \leq z < y} \mu(x,z), & \forall x<y \in P. 
\end{eqnarray*}

The \emph{Moebius invariant}, or \emph{Moebius number}, of a bounded poset $P$ is defined as: 
\begin{equation*}
\mu(P):=\mu(\hat{0}_P, \hat{1}_P). 
\end{equation*}

\end{defi}

\begin{exple} The Moebius number of the poset $B_n$ of subsets of $\intunn$, ordered by inclusion, is $(-1)^n$. 
\end{exple}
	\subsection{Generalities on incidence Hopf algebra}

	All the definitions recalled here are extracted from the article of W. Schmitt \cite{IHA}. 

	 A family of posets $\mathcal{P}$ is \emph{interval closed}, if it is non-empty and, for all $P \in \mathcal{P}$ and $x \leq y \in P$, the interval $[x,y]$ belongs to $\mathcal{P}$. An \emph{order compatible relation} on an interval closed family $\mathcal{P}$ is an equivalence relation $\sim$ such that $P \sim Q$ if and only if there exists a bijection $\phi: P \rightarrow Q$ such that $[0_P,x] \sim [0_Q,\phi(x)]$ and $[x, 1_P] \sim [\phi(x), 1_Q]$, for all $x \in \mathcal{P}$. The isomorphism of posets is an example of order compatible relation. 
	
	Given $K$ a commutative ring with a unit, and $\sim$ an order compatible relation on an interval closed family $\mathcal{P}$, we consider the quotient set $\mathcal{P}/\sim$ and denote by $[P]$ the $\sim$-equivalence class of a poset $P \in \mathcal{P}$. We define a $K$-coalgebra $\CP$ as follow: 
	
\begin{prop}[Theorem 3.1 in \cite{IHA}]
Let $\CP$ denote the free $K$-module generated by $\mathcal{P}/\sim$. We define linear maps $\Delta: \CP \rightarrow \CP \otimes \CP$ and $\epsilon: \CP \rightarrow K$ by: 
\begin{equation*}
\Delta[P]=\sum_{x \in P} [0_P,x] \otimes [x,1_P]
\end{equation*}
and
\begin{equation*}
\epsilon[P]=\delta_{|P|,1},
\end{equation*}
where $\delta_{i,j}$ is the Kronecker symbol. 
Then, $\CP$ is a coalgebra with comultiplication $\Delta$ and counit $\epsilon$. 
\end{prop}

	The \emph{direct product} of posets $P_1$ and $P_2$ is the cartesian product $P_1 \times P_2$ partially ordered by the relation $(x_1,x_2) \leq (y_1,y_2)$ if and only if $x_i \leq y_i$ in $P_i$, for $i=1,2$. A \emph{hereditary family} is an interval closed family which is also closed under formation of direct products. Let $\sim$ be an order compatible relation on $\mathcal{P}$ which is also a semigroup congruence, i.e., whenever $P \sim Q$ in $\mathcal{P}$, then $P \times R \sim Q \times R$ and $R \times P \sim R \times Q$, for all $R \in \mathcal{P}$. This relation is \emph{reduced} if whenever $|R|=1$, then $P \times R \sim R \times P \sim P$. These hypotheses assure that product will be well defined on the quotient. An order compatible relation on a hereditary family $\mathcal{P}$ which is also a reduced congruence is called a \emph{Hopf relation} on $\mathcal{P}$. The isomorphism of posets is a Hopf relation.
	
	\begin{prop}[\cite{AIC}] Let $\sim$ be a Hopf relation on a hereditary family $\mathcal{P}$. Then $\HP=(\CP, \times, \Delta, \epsilon,S)$ is a Hopf algebra over $K$. 
	\end{prop}

\begin{exple}
The incidence Hopf algebra generated by the family of poset of subsets of $\intunn$ is the polynomial algebra $K[x]$, endowed with the following coproduct: 
\begin{equation*}
\Delta(x^n)=\sum_{k=0}^{n} \binom{n}{k} x^k \otimes x^{n-k}. 
\end{equation*}
\end{exple}

We deal here with the field $K=\mathbb{Q}$. 

We can consider the set of $\mathbb{Q}$-linear homomorphisms of algebras between $\HP$ and $\mathbb{Q}$, which send the trivial poset to the unit of $\mathbb{Q}$. These homomorphisms are called \emph{characters}. The set of characters can be endowed with a structure of group as follows. Given two characters $\phi$ and $\psi$, the convolution of $\phi$ and $\psi$ is defined on any element $P$ of $\HP$ by:
\begin{equation*}
\phi \ast \psi (P)=\sum \phi(P_{(1)})\psi(P_{(2)}),
\end{equation*}
where $\Delta (P) = \sum P_{(1)} \otimes P_{(2)} $, using Sweedler's convention. The unit of this group is the counit of the Hopf algebra $\HP$. 

		\section{Incidence Hopf algebra of triangle and diamond posets}

			\subsection{Presentation of the triangle and diamond posets and their incidence Hopf algebra}

\label{DP}
Let us consider the family $F_0$ generated by posets $\{(d_i)_{i \geq 1}, (t_j)_{j \geq 3}\}$, such that $d_1$ is the trivial poset, $d_i$ is an interval for all $i \geq 2$ and $t_j$ is a poset with a least element but without a greatest one. The $(d_i)_{i \geq 1}$ will be called the diamond posets and the $(t_j)_{j \geq 2}$ will be called the triangle posets. We denote by $\widehat{t_j}$ the augmented triangle posets, bounded by the addition of a greatest element $\hat{1}$. We moreover assume that: 
\begin{itemize}
\item any closed interval in a diamond poset can be written as a product of diamond posets, (Decomposition Property 1)
\item any closed interval in a triangle poset can be written as a product of diamond posets, (Decomposition Property 2 a)
\item and any half-open interval $[t,\hat{1})$ in an augmented triangle poset $\widehat{t_j}$ can be written as a product of triangle or trivial posets. (Decomposition Property 2 b)
\end{itemize}

We denote by $F_1$ the hereditary family generated by $F_0$: due to Decomposition properties this family is constituted by direct products of diamond and triangle posets. As diamond and triangle posets admit a least element, all posets in $F_1$ have a least element, some of them are intervals but others are not. We construct a hereditary family of intervals from this family.

To apply Schmitt's construction, we now consider the hereditary family $F_2$ generated by the elements of the family $F_1$, augmented with a maximal element when they are not intervals: $F_2$ is then a hereditary family of intervals. We apply Schmitt's construction to this family taking the isomorphism of posets as a reduced order compatible relation to obtain the incidence Hopf algebra $\mathcal{H}_{\diamond, \triangledown}$. We show that, under some assumptions, the calculus of some characters on $\mathcal{H}_{\diamond, \triangledown}$ can be reduced to a calculus in a smaller algebra. 

			\subsection{A smaller bialgebra constructed on triangle and diamond posets} \label{bialg}

The family $F_1$ is closed under the direct product and interval closed, in the sense that any closed or half-open interval of a poset of the family belongs to the family. We construct a bialgebra from this family in the same way as Schmitt constructs an incidence Hopf algebra from a hereditary family of intervals.

Taking the isomorphism of posets as a Hopf relation $\sim$, the set $\tilde{F_1}=F_1/\sim$ is a monoid, with product induced by direct product of posets and identity element $1$ equal to the class of any one point interval. Let us denote by $V(F_1)$ the free $K$-module generated by $\tilde{F_1}$. The structure of monoid on $\tilde{F_1}$ induces a structure of algebra on $V(F_1)$, isomorphic to the monoid algebra of $\tilde{F_1}$ over $K$. As $F_1$ is the set of monomials on triangle and diamond posets of $F_0$, the algebra $V(F_1)$ is generated by isomorphism classes of triangle and diamond posets of $F_0$. All elements of $F_1$ have a least element. We endow it with the following coproduct defined on $d$, an isomorphism class of posets with both a least and a greatest elements, and $t$, an isomorphism class of 
posets with a least element but no greatest one, by:
\begin{equation*}
\Delta(d)=\sum_{x \in d} [\hat{0}_{d}, x] \otimes [x,\hat{1}_{d}], 
\end{equation*}
and 
\begin{equation*}
\Delta(t)= \sum _{x \in t} [\hat{0}_{t}, x] \otimes [x,\hat{1}_{\widehat{t}}).
\end{equation*}
This coproduct is a morphism of algebras, as an interval in a product of posets can be seen as a product of intervals.

Let us remark that diamond posets are intervals whereas triangle posets are not: these two types of posets thus cannot belong to the same isomorphism class. We denote by $\mathcal{B}_{\diamond, \triangledown}$ the obtained bialgebra.

\medskip

The subalgebra $\mathcal{D}_\diamond$ generated by the diamond posets is also a subcoalgebra according to the first decomposition property: this is a subbialgebra of $\mathcal{B}_{\diamond, \triangledown}$. This subbialgebra is isomorphic as a bialgebra to the incidence Hopf algebra of diamond posets. According to the second decomposition property, the subalgebra $\mathcal{T}_\triangledown$ of $\mathcal{B}_{\diamond, \triangledown}$ generated by the triangle posets is a right comodule over $\mathcal{D}_\diamond$. Thanks to the definition of the coproduct and the structure of direct product of posets, the coproduct on $\mathcal{T}_\triangledown$ is still a homomorphism of algebras. 

\begin{figure}
\begin{tikzpicture}[scale=0. 5]
\draw (-2,0) -- (0,-2) -- (2,0) -- (0,2) -- (-2,0);
\draw[fill=gray] (-1. 5,-0. 5) -- (0,-2) -- (1. 5,-0. 5) -- (0,1) -- (-1. 5,-0. 5);
\draw[fill=gray] (0,1) -- (0. 5,1. 5) -- (0,2) -- (-0. 5,1. 5) -- (0,1);
\draw(0,1) node{$\bullet$};
\end{tikzpicture}
\begin{tikzpicture}[scale=0. 5]
\draw (-3,1) -- (0,-2) -- (3,1) -- (-3,1);
\draw[fill=gray] (-1,-1) -- (0,-2) -- (1,-1) -- (0,0) -- (-1,-1);
\draw[fill=gray] (0,0) -- (1,1) -- (-1,1) -- (0,0);
\draw(0,0) node{$\bullet$};
\end{tikzpicture}
\caption{Intervals in diamond and triangle posets: all of them are products of diamond posets, except half-open upper intervals in triangle posets which are products of triangle posets.}
\end{figure}
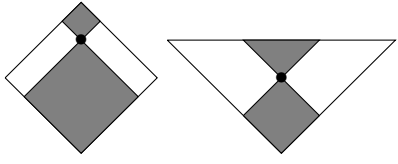

We show that the computation of some characters on the incidence Hopf algebra $\mathcal{H}_{\diamond, \triangledown}$ can be reduced to some calculus on the bialgebra $\mathcal{B}_{\diamond, \triangledown}$. 

			\subsection{Results on the computation of some characters on the incidence Hopf algebra of triangle and diamond posets}

We define the following linear application: 
\begin{equation*}
\lambda: \mathcal{B}_{\diamond, \triangledown} \rightarrow \mathcal{H}_{\diamond, \triangledown}
\end{equation*}
which sends an isomorphism class $c_{d_i}$ of a diamond poset $d_i$ in $\mathcal{B}_{\diamond, \triangledown}$ to the isomorphism class $c_{d_i}$ of $d_i$ in $\mathcal{H}_{\diamond, \triangledown}$ and which sends an isomorphism class $c_{t_j}$ of a triangle poset $t_j$ in $\mathcal{B}_{\diamond, \triangledown}$ to the isomorphism class $c_{\widehat{t_j}}$ of the augmented triangle poset $\widehat{t_j}$ in $\mathcal{H}_{\diamond, \triangledown}$.

This linear application is well defined. Indeed, if two diamond posets are in the same isomorphism class in $\mathcal{B}_{\diamond, \triangledown}$, then they are isomorphic so they are also in the same isomorphism class in $\mathcal{H}_{\diamond, \triangledown}$. If two triangle posets are in the same isomorphism class in $\mathcal{B}_{\diamond, \triangledown}$, then they are isomorphic so these posets augmented with a greatest element are also isomorphic, and thus in the same isomorphism class in $\mathcal{H}_{\diamond, \triangledown}$.

We would like to compute some characters on the isomorphism classes of the diamond posets $d_i$ of $F_0$ and augmented triangle posets $\widehat{t_j}$ coming from triangle posets $t_j$ of $F_0$ in $\mathcal{H}_{\diamond, \triangledown}$. As $F_0$ is a subfamily of $F_1$, to any element of $F_0$ corresponds an isomorphism class in $\mathcal{B}_{\diamond, \triangledown}$ which is sent to the isomorphism class of the corresponding element in $\mathcal{H}_{\diamond, \triangledown}$ : the elements on which we want to compute characters belong to the image of $\lambda$. 

We moreover remark that the fibre of an isomorphism class in the image of $\lambda$ is made of at most one isomorphism class of triangle posets and at most one isomorphism class of diamond poset. Indeed, if two isomorphism classes of triangle posets, or two isomorphism classes of diamond posets, are sent by $\lambda$ to the same isomorphism class, then these isomorphism classes are equal. 

Let us consider two characters $\alpha$ and $\beta$ on $\mathcal{H}_{\diamond, \triangledown}$ such that there exists two rational  numbers $\epsilon_\alpha$, $\epsilon_\beta$ and two maps $\widetilde{\alpha}$ and $\widetilde{\beta}$ from $\mathcal{B}_{\diamond, \triangledown}$ to $\mathbb{Q}$ which satisfy: 
\begin{equation*}
\alpha (\lambda(c_{d_i})) = \widetilde{\alpha} (c_{d_i}), \text{ } \alpha (\lambda(c_{t_j})) = \epsilon_\alpha \widetilde{\alpha} (c_{t_j}),
\end{equation*}
and
\begin{equation*}
\beta (\lambda(c_{d_i})) = \widetilde{\beta} (c_{d_i}), \text{ } \beta (\lambda(c_{t_j})) = \epsilon_\beta \widetilde{\beta} (c_{t_j}),
\end{equation*}
for all isomorphism class $c_{d_i}$ of diamond poset $d_i$ of $F_0$ and $c_{t_j}$ of triangle poset $t_j$ of $F_0$.

Then, the convolution of $\alpha$ and $\beta$ can be computed thanks to the following theorem:
\begin{thm} \label{thmconv}
The convolution of the characters $\alpha$ and $\beta$ on $\mathcal{H}_{\diamond, \triangledown}$ is given by:
\begin{equation*}
\alpha \ast \beta (\lambda(c_{d_i})) = \sum \widetilde{\alpha}(c_{d_i}^{(1)}) \widetilde{\beta}(c_{d_i}^{(2)}),
\end{equation*}
and 
\begin{equation*}
\alpha \ast \beta (\lambda(c_{t_j})) = \epsilon_\beta \sum \widetilde{\alpha}(c_{t_j}^{(1)}) \widetilde{\beta}(c_{t_j}^{(2)}) + \epsilon_\alpha \widetilde{\alpha}(c_{t_j}),
\end{equation*}
where $\Delta(c_{d_i})=\sum c_{d_i}^{(1)} \otimes c_{d_i}^{(2)}$ and $\Delta(c_{t_j})=\sum c_{t_j}^{(1)} \otimes c_{t_j}^{(2)}$ in $\mathcal{B}_{\diamond, \triangledown}$. 
\end{thm} 

\begin{proof}
The isomorphism class $\lambda(c_{d_i})$ is the isomorphism class of the diamond poset $d_i$ in $\mathcal{H}_{\diamond, \triangledown}$ by definition of $\lambda$. Moreover, the coproduct of $\lambda(c_{d_i})$ in $\mathcal{H}_{\diamond, \triangledown}$ and of $c_{d_i}$ in $\mathcal{B}_{\diamond, \triangledown}$ are the same by definition of the coproduct. As $\alpha$, $\widetilde{\alpha}$ on the one hand and $\beta$ and $\widetilde{\beta}$ on the other hand are equal on $\lambda(c_{d_i})$ and $c_{d_i}$ respectively, the first equality follows. 

To obtain the second equality, let us remark that the isomorphism class $\lambda(c_{t_j})$ corresponds to the isomorphism class of the augmented triangle poset $\widehat{t}_j$ by definition of $\lambda$. Hence the coproduct of $\lambda(c_{t_j})$ in $\mathcal{H}_{\diamond, \triangledown}$ has one more term than the coproduct of $c_{t_j}$ on $\mathcal{B}_{\diamond, \triangledown}$, due to the fact that the poset $\widehat{t_j}$ has one more element than the poset $t_j$. This term is $ \lambda(c_{t_j}) \otimes 1$. All the other terms can be matched by associating $\lambda(c_{t_j})$ to the unique isomorphism class of triangle poset of its fibre $c_{t_j}$. Moreover, the posets on the left part of the coproduct of $\lambda(c_{t_j})$ are of diamond type, except for the term that does not belong to the coproduct of $c_{t_j}$ and $\alpha$ and $\widetilde{\alpha}$ coincides on diamond posets. Therefore we have: 
\begin{equation*}
\alpha \ast \beta (\lambda(c_{t_j}))= \sum \widetilde{\alpha}(c_{t_j}^{(1)}) \epsilon_\beta \widetilde{\beta}(c_{t_j}^{(2)}) + \epsilon_\alpha \widetilde{\alpha}(c_{t_j}). 
\end{equation*}
This gives the result. 
\end{proof}

Remark that as the assumptions on $\alpha$ and $\beta$ are the same, the previous theorem also gives the formula for $\beta \ast \alpha$. 

We will use these results in Section \ref{compmoeb} to compute some characters on the hypertree posets.

	\section{Incidence Hopf algebra of hypertree posets and partition posets}

From now on, we choose $K=\mathbb{Q}$.
 
		\subsection{Incidence Hopf algebra of the hypertree posets}
			
		A \emph{hypergraph} is a pair $(V,E)$, where the elements of $V$ are called \emph{vertices} and the elements of $E$, called \emph{edges}, are sets of at least two vertices. The \emph{size} of an edge $e$ is the number of vertices in the edge $e$. The \emph{valency} of a vertex $v$ is the number of edges to which $v$ belongs. A \emph{walk} on a hypergraph $H=(V,E)$ from a vertex $s$ of $H$ to a vertex $f$ of $H$ is an alternating sequence of vertices and edges in $H$ $(s=v_0, e_0, v_1, e_1, \dots, e_{n-1}, v_n=f)$ such that $e_i$ are edges containing the vertices $v_i$ and $v_{i+1}$, for all $i \in \llbracket 0,n-1 \rrbracket $. A \emph{hypertree} is a hypergraph such that given any pair $(s,f)$ of vertices, there exists one and only one walk from $s$ to $f$ without repeated edges. We say that a hypertree is on $n$ vertices if the set $V$ is of cardinality $n$. 
		
		We can define the following order on hypertrees: a hypertree $T$ is smaller than a hypertree $T'$ whenever the edges of $T$ are unions of some edges of $T'$. The set of hypertrees on $n$ vertices endowed with this partial order is a poset denoted by $\htn$. This poset has a least element $\hat{0}$ which is the hypertree with only one edge. The poset obtained by adding to this poset a greatest element $\hat{1}$ is called the $(n-)$ augmented hypertree poset and denoted by $\HTn$. 
		
\begin{figure}[h!]
\begin{tikzpicture}[scale=0. 80]
\draw[black, fill=gray!40] (0,0) -- (1,0) -- (1,1) -- (0,0);
\draw (-2,3) -- (-3,3) -- (-2,4);
\draw (0,3) -- (1,4) -- (1,3);
\draw (3,3) -- (4,3) -- (4,4);
\draw[thick] (-2. 5,2. 8) -- (0. 5,1. 5);
\draw[thick] (0. 5,2. 8) -- (0. 5,1. 5);
\draw[thick] (3. 5,2. 8) -- (0. 5,1. 5);
\draw[thick] (-2. 5,4. 2) -- (0. 5,5. 5);
\draw[thick] (0. 5,4. 2) -- (0. 5,5. 5);
\draw[thick] (3. 5,4. 2) -- (0. 5,5. 5);
\draw[black, fill=white] (0,0) circle (0. 2);
\draw[black, fill=white] (1,0) circle (0. 2);
\draw[black, fill=white] (1,1) circle (0. 2);
\draw[black, fill=white] (-2,3) circle (0. 2);
\draw[black, fill=white] (-3,3) circle (0. 2);
\draw[black, fill=white] (-2,4) circle (0. 2);
\draw[black, fill=white] (0,3) circle (0. 2);
\draw[black, fill=white] (1,4) circle (0. 2);
\draw[black, fill=white] (1,3) circle (0. 2);
\draw[black, fill=white] (3,3) circle (0. 2);
\draw[black, fill=white] (4,3) circle (0. 2);
\draw[black, fill=white] (4,4) circle (0. 2);
\draw(0. 5,6) node{$\hat{1}$};
\draw(0,0) node{$1$};
\draw(1,1) node{$2$};
\draw(1,0) node{$3$};
\draw(-3,3) node{$1$};
\draw(-2,4) node{$2$};
\draw(-2,3) node{$3$};
\draw(0,3) node{$1$};
\draw(1,4) node{$2$};
\draw(1,3) node{$3$};
\draw(3,3) node{$1$};
\draw(4,4) node{$2$};
\draw(4,3) node{$3$};
\end{tikzpicture}
\caption{The poset $\widehat{HT_3}$}\label{HT3}
\end{figure}
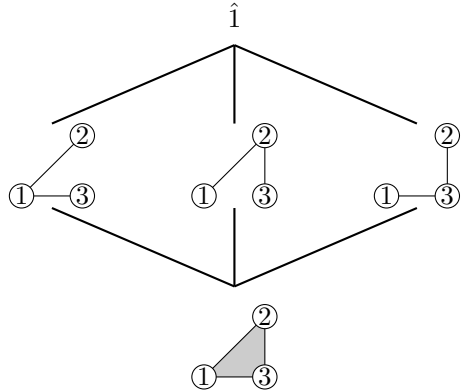		
		
		Some intervals in the hypertree posets will be described in terms of another type of posets: the partition posets. A \emph{partition poset} is a poset on the set of all the partitions of a set $V$. A partition $p_1$ is smaller than another one $p_2$ if each part of $p_1$ is the union of some parts of $p_2$. The partition poset on $n$ vertices $\Pi_n$ is based on the set of partitions of a set of cardinality $n$. 

\begin{center}
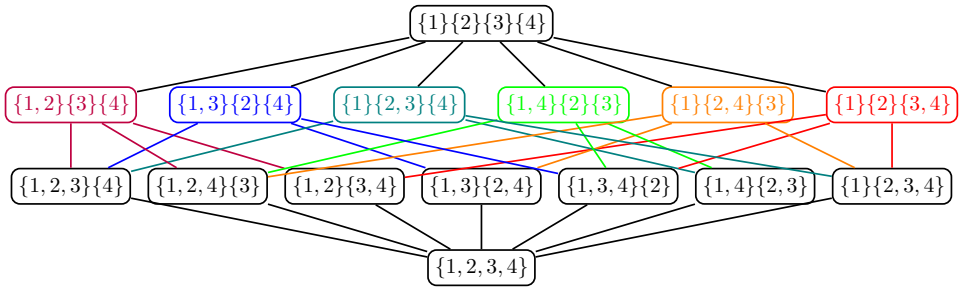
\begin{figure}[h!]
\scalebox{0.8}
{\begin{tikzpicture}[scale=0.9]
\node[draw, thick, rounded corners] (max) at (0,4. 5) {$\{1\}\{2\}\{3\}\{4\}$};
\node[draw, thick, rounded corners, purple] (12) at (-7.5,3) {$\{1, 2\}\{3\}\{4\}$};
\node[draw, thick, rounded corners, blue] (13) at (-4.5,3) {$\{1, 3\}\{2\}\{4\}$};
\node[draw, thick, rounded corners, teal] (23) at (-1. 5,3) {$\{1\}\{2, 3\}\{4\}$};
\node[draw, thick, rounded corners, green] (14) at (1.5,3) {$\{1, 4\}\{2\}\{3\}$};
\node[draw, thick, rounded corners, orange] (24) at (4.5,3) {$\{1\}\{2,4\}\{3\}$};
\node[draw, thick, rounded corners, red] (34) at (7.5,3) {$\{1\}\{2\}\{3, 4\}$};
\draw[thick] (max) -- (12);
\draw[thick] (max) -- (13);
\draw[thick] (max) -- (14);
\draw[thick] (max) -- (23);
\draw[thick] (max) -- (24);
\draw[thick] (max) -- (34);
\node[draw, thick, rounded corners] (123) at (-7.5,1. 5) {$\{1, 2, 3\}\{4\}$};
\node[draw, thick, rounded corners] (124) at (-5,1. 5) {$\{1, 2, 4\}\{3\}$};
\node[draw, thick, rounded corners] (12d) at (-2.5,1. 5) {$\{1, 2\}\{3, 4\}$};
\node[draw, thick, rounded corners] (13d) at (0,1. 5) {$\{1, 3\}\{2, 4\}$};
\node[draw, thick, rounded corners] (134) at (2.5,1. 5) {$\{1, 3, 4\}\{2\}$};
\node[draw, thick, rounded corners] (14d) at (5,1. 5) {$\{1, 4\}\{2, 3\}$};
\node[draw, thick, rounded corners] (234) at (7.5,1. 5) {$\{1\}\{2, 3, 4\}$};
\draw[thick, purple] (12) -- (12d);
\draw[thick, red] (34) -- (12d);
\draw[thick, blue] (13) -- (13d);
\draw[thick, orange] (24) -- (13d);
\draw[thick, green] (14) -- (14d);
\draw[thick, teal] (23) -- (14d);
\draw[thick, purple] (12) -- (123);
\draw[thick, blue] (13) -- (123);
\draw[thick, teal] (23) -- (123);
\draw[thick, purple] (12) -- (124);
\draw[thick, green] (14) -- (124);
\draw[thick, orange] (24) -- (124);
\draw[thick, blue] (13) -- (134);
\draw[thick, green] (14) -- (134);
\draw[thick, red] (34) -- (134);
\draw[thick, teal] (23) -- (234);
\draw[thick, orange] (24) -- (234);
\draw[thick, red] (34) -- (234);
\node[draw, thick, rounded corners] (min) at (0,0) {$\{1, 2, 3, 4\}$};
\draw[thick] (min) -- (12d);
\draw[thick] (min) -- (13d);
\draw[thick] (min) -- (14d);
\draw[thick] (min) -- (123);
\draw[thick] (min) -- (124);
\draw[thick] (min) -- (134);
\draw[thick] (min) -- (234);
\end{tikzpicture}}
\caption{The poset $\Pi_4$}
\end{figure}	
\end{center}
		
		We need the following result of J. McCammond and J. Meier on intervals in the hypertree poset: 
		
		\begin{lem}[Lemma 2.5,\cite{McCM}]\label{LMcCM} Let $\tau$ be a hypertree on $n$ vertices. 
		
\begin{enumerate}[(a)]
\item The interval $[\hat{0},\tau]$ is a direct product of partition posets, with one factor $\Pi_j$ for each vertex in $\tau$ with valency $j$. 
\item The half-open interval $[\tau,\hat{1})$ is a direct product of hypertree posets, with one factor $\operatorname{HT_j}$ for each edge in $\tau$ with size $j$. 
\end{enumerate}
	
		\end{lem}

Let us consider the incidence Hopf algebra $\HHTn = (\HHTn,\times, \epsilon, \eta, \Delta, S)$ obtained from the construction of \S \ref{DP} by taking the set of partition posets $(p_i)_{i \geq 1}$ for the set of diamond posets and the set of hypertree posets $(h_n)_{n \geq 3}$, where $h_n$ is the shorter notation for $HT_n$, for the set of triangle posets. Indeed, $p_1$, the partition poset on one element and the hypertree poset on two elements $HT_2$ are isomorphic to the trivial poset, partition posets are intervals and hypertree posets have a least element but no greatest one. Moreover, it is a classical result that every interval in a partition poset is isomorphic to a product of partition posets. This fact combined with Lemma \ref{LMcCM} implies that this family satisfies the decomposition property and then all the requirements of Section \ref{DP}. We will also denote by $\widehat{h_n}$ the augmented hypertree poset $\widehat{HT_n}$.

We consider $\HHTn^*$, the group of characters $\chi: \HHTn \rightarrow \mathbb{Q}$. We aim at calculating the Moebius numbers for the augmented hypertree posets using the classical techniques of characters. A good reference for such a computation of characters, and Moebius numbers, for the partition posets is the article \cite{Speicher1997} of R. Speicher. To compute the character which associates to any poset of $\HHTn$ its Moebius number, we use Theorem \ref{thmconv}. 

We call $\Spnr$ the bialgebra defined in Section \ref{bialg}. Thanks to Lemma \ref{LMcCM}(b), we obtain that this bialgebra is not only generated as an algebra by isomorphism classes of partition posets and isomorphism classes of  intervals $[\tau,\hat{1})$, for any hypertree $\tau$, but also by a smaller set: the isomorphism classes of partition posets $p_n$ and the isomorphism classes of hypertree posets $h_n$. Moreover, partition posets and hypertree posets are both graded, therefore two partition posets or hypertree posets respectively on $n$ and $m$ are isomorphic if and only if $m$ and $n$ are equal. As every $p_i$ and $h_j$ are pairwise in different isomorphism classes, due to gradings, and as we focus on these classes, we will use the same notation for the isomorphism classes of posets and posets themselves.

Hence, the convolution of characters $\alpha$ and $\beta$ on $\HHTn$ can be computed using the bialgebra $\Spnr$: 

\begin{prop} \label{propdec} The convolution of characters $\alpha$ and $\beta$ on $\HHTn$ can be computed using maps $\widetilde{\alpha}$ and $\widetilde{\beta}$ from $\Spnr$ to $\mathbb{Q}$, provided they exist and satisfy the following equations, for all $i \geq 1$ and $n \geq 3$: 
\begin{equation*}
\alpha (p_i) = \widetilde{\alpha} (p_i), \text{ } \alpha (\widehat{h}_n) = \epsilon_\alpha \widetilde{\alpha} (h_n),
\end{equation*}
\begin{equation*}
\beta (p_i) = \widetilde{\beta} (p_i), \text{ } \beta (\widehat{h}_n) = \epsilon_\beta \widetilde{\beta} (h_n),
\end{equation*}
with $\epsilon_\alpha, \epsilon_\beta \in \mathbb{Q}$. 

This computation is given by: 
\begin{equation*}
\alpha \ast \beta (p_i) = \sum \widetilde{\alpha}(p_i^{(1)}) \widetilde{\beta}(p_i^{(2)}),
\end{equation*}
and 
\begin{equation*}
\alpha \ast \beta (\widehat{h_j}) = \epsilon_\beta \sum \widetilde{\alpha}(h_j^{(1)}) \widetilde{\beta}(h_j^{(2)}) + \epsilon_\alpha \widetilde{\alpha}(h_j),
\end{equation*}
where $\Delta(p_i)=\sum p_i^{(1)} \otimes p_i^{(2)}$ and $\Delta(h_j)=\sum h_j^{(1)} \otimes h_j^{(2)}$ in $\Spnr$. 
\end{prop}
	
	\begin{proof}
This is a corollary of Theorem \ref{thmconv} for $p_i$ and $h_j$. 
	\end{proof}
	
\begin{exple} We consider the poset $\widehat{HT_3}$ drawn on figure \ref{HT3}. 

The computation of the coproducts gives:
\begin{align*}
\Delta (\widehat{h_3}) = 1 \otimes \widehat{h_3} + 3 \  p_2 \otimes \widehat{h_2} + \widehat{h_3} \otimes 1 \text{, in $\HHTn$,}\\
\Delta (h_3)= 1 \otimes h_3 + 3 \ p_2 \otimes h_2 \text{, in $\Spnr$} .
\end{align*}
\end{exple}

We determine a closed expression for the coproduct in the next section. 

		\subsection{Computation of the coproduct}

	We now compute the coproduct $\Dr$ in the algebra $\Spnr$. We denote by $1$ the neutral element of $\Spnr$ for the product, i.e. the trivial poset. 
	
	The coproduct of isomorphism classes of partition posets $p_n$ has already been computed. It can be found for instance in the article of W. Schmitt \cite{IHA}:

	\begin{prop}[Example 14.1 in \cite{IHA}]:
	The coproduct on the isomorphism classes of partition posets is given by: 
		\begin{equation*}
	\Dr \left( \frac{p_n}{n!} \right) = \sum_{k=1}^{n} \sum_{\substack{(j_1, \ldots, j_n) \in \mathbb{N}\\ \sum_{i=1}^n j_i=k, \sum_{i=1}^n i j_i=n}} \binom{k}{j_1, \dots, j_n} \prod_{i=1}^n \left(\frac{p_i}{i!}\right)^{j_i}  \otimes \frac{p_k}{k!},
	\end{equation*}
	\end{prop}
	where $p_1$ is the trivial poset. 

	Let us now compute the coproduct for $\hn$. According to the structure of the hypertree posets and Lemma \ref{LMcCM}, the left part of the coproduct of isomorphism classes of the hypertree poset $h_n$ is a product of isomorphism classes of partition posets and the right part is a product of isomorphism classes of hypertree posets $h_k$. We first establish a criterion to describe the tensor products appearing in the coproduct of $\hn$. We write the coproduct as follows:
	\begin{equation}\label{defCoprod}
	\Dr (\hn)=\sum_{(\alpha, \pi) \in \mathcal{P}_n} \cnap p_\alpha \otimes h_\pi,
	\end{equation}
where $\mathcal{P}_n$ is the set of pairs $(\alpha, \pi)$ such that $\cnap$ does not vanish, and for all $\alpha=(\alpha_1,\alpha_2, \dots, \alpha_k)$ and $\pi=(\pi_2, \pi_3, \dots, \pi_l)$, $p_\alpha=1^{\alpha_1} p_2^{\alpha_2} \dots p_k^{\alpha_k}$ and $h_\pi=h_2^{\pi_2}h_3^{\pi_3}\dots h_l^{\pi_l}$. The coefficient $\cnap$ corresponds to the number of hypertrees in $\htn$ with $\alpha_i$ vertices of valency $i$ and $\pi_j$ edges of size $j$, for all $i \geq 1$ and $j \geq 2$. 

We now characterize the set $\mathcal{P}_n$. We consider hypertrees as $\intunn$-labelled bipartite trees as in \cite{McCuMi}. A $\intunn$-labelled bipartite tree is a tree $T$ together with a bijection from $\intunn$ to a subset of its vertex set such that the image of $\intunn$ includes all of the vertices of valency $1$ and for every edge in $T$ exactly one of its endpoints lies in the image of $\intunn$. The labelled vertices of a bipartite tree correspond to the vertices of the associated hypertree and the other vertices correspond to the edges of the hypertree. We denote by $\alpha_i$ the number of labelled vertices of valency $i$ and by $\pi_j$ the number of unlabelled vertices of valency $j$ (or of edges of size $j$ in the hypertree). We want to determine necessary and sufficient conditions on $(\alpha_1, \dots, \alpha_l)$ and $(\pi_2, \dots, \pi_k)$ for the existence of a hypertree with $\alpha_i$ vertices of valency $i$ and $\pi_j$ edges of size $j$, for all $i \geq 1$ and $j \geq 2$. 

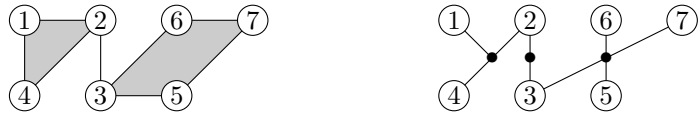
\begin{figure}[h!]
\begin{tikzpicture}
\draw[black, fill=gray!40] (0,0) -- (1,1) -- (0,1) -- (0,0);
\draw[black, fill=gray!40] (1,0) -- (2,0) -- (3,1) -- (2,1) -- (1,0);
\draw(1,0) -- (1,1);
\draw[black, fill=white] (0,0) circle (0.2);
\draw[black, fill=white] (0,1) circle (0.2);
\draw[black, fill=white] (1,1) circle (0.2);
\draw[black, fill=white] (1,0) circle (0.2);
\draw[black, fill=white] (2,0) circle (0.2);
\draw[black, fill=white] (3,1) circle (0.2);
\draw[black, fill=white] (2,1) circle (0.2);
\draw(0,0) node{$4$};
\draw(0,1) node{$1$};
\draw(1,1) node{$2$};
\draw(1,0) node{$3$};
\draw(2,0) node{$5$};
\draw(2,1) node{$6$};
\draw(3,1) node{$7$};
\end{tikzpicture} \hspace{2cm}
\begin{tikzpicture}
\draw(0.5,0.5) -- (1,1);
\draw(0.5,0.5) -- (0,1);
\draw(0.5,0.5) -- (0,0);
\draw(1,1) -- (1,0);
\draw(2,1) -- (2,0);
\draw(1,0) -- (3,1);
\draw[black, fill=white] (0,0) circle (0.2);
\draw[black, fill=white] (0,1) circle (0.2);
\draw[black, fill=white] (1,1) circle (0.2);
\draw[black, fill=white] (1,0) circle (0.2);
\draw[black, fill=white] (2,0) circle (0.2);
\draw[black, fill=white] (3,1) circle (0.2);
\draw[black, fill=white] (2,1) circle (0.2);
\draw(0.5,0.5) node{$\bullet$};
\draw(1,0.5) node{$\bullet$};
\draw(2,0.5) node{$\bullet$};
\draw(0,0) node{$4$};
\draw(0,1) node{$1$};
\draw(1,1) node{$2$};
\draw(1,0) node{$3$};
\draw(2,0) node{$5$};
\draw(2,1) node{$6$};
\draw(3,1) node{$7$};
\end{tikzpicture}
\caption{A hypertree and its associated labelled bipartite tree.}
\end{figure}

We hence obtain the following criterion for the non-vanishing of $\cnap$, expressed in terms of hypertrees: 

\begin{prop} Given two tuples $\alpha=(\alpha_1, \ldots, \alpha_k)$ and $\pi=(\pi_2, \ldots, \pi_l)$, there exists a hypertree with $\alpha_i$ vertices of valency $i$ and $\pi_j$ edges of size $j$ if and only if:
 \begin{equation} \label{theeq}
\sum_{i=1}^k \alpha_i = n, \quad \sum_{j=2}^{l} (j-1) \pi_j = n-1 \text{ and } \sum_{i=1}^k i \alpha_i = n + \sum_{j=2}^{l} \pi_j -1.
\end{equation}

\end{prop}

We postpone the proof of this proposition to illustrate it through an example. 
\begin{exple}\label{criter}
For $n=4$, the second equation of \eqref{theeq} implies that $\sum_{j=2}^{l} \pi_j \leq 3$, i.e. $\sum_{i=1}^k i \alpha_i \leq 6 $. The possible $\alpha$ are: 
\begin{itemize}
\item $\alpha=(4)$, then we obtain the condition $\sum_{j=2}^{l} \pi_j=1$ so the only possible $\pi$ is $\pi=(0,0,1)$,
\item $\alpha=(3,1)$, then we obtain the condition $\sum_{j=2}^{l} \pi_j=2$ so the only possible $\pi$ is $\pi=(1,1)$,
\item $\alpha=(2,2)$, then we obtain the condition $\sum_{j=2}^{l} \pi_j=3$ so the only possible $\pi$ is $\pi=(3)$,
\item $\alpha=(3,0,1)$, then we obtain the condition $\sum_{j=2}^{l} \pi_j=3$ so the only possible $\pi$ is $\pi=(3)$. 
\end{itemize}
\end{exple}

\begin{proof}
	Suppose that there exists such a hypertree. Every vertex has a fixed valency. Therefore, counting vertices, we have the first equation: 
\begin{equation*}
\sum_{i=1}^k \alpha_i = n.
\end{equation*}
By construction of the labelled tree, every unlabelled vertex is linked with a labelled vertex. This leads to the following equality by counting edges around labelled and unlabelled vertices: 
\begin{equation}\label{eqalpi}
\sum_{i=1}^k i \alpha_i = \sum_{j=2}^l j \pi_j. 
\end{equation}
Moreover, to a bipartite tree can be associated a simplicial complex with faces of dimension at most $1$. This simplicial complex is connected without cycles, therefore its Euler characteristic is equal to $1$ and can be expressed as: 
\begin{equation} \label{Eulercar}
\chi=1=\sum_{j\geq 2} \pi_j - \sum_{j\geq 2} j \pi_j + \sum_{i\geq 1} \alpha_i . 
\end{equation}
These equations are equivalent to Equations \eqref{theeq}. 

We can also deduce from the second equation of the proposition the following expression of $\pi_2$ in terms of $\pi_j$ for $j \geq 3$: 
	\begin{equation*}
	\pi_2 = n -1 - \sum_{j \geq 3} (j-1) \pi_j. 
	\end{equation*}

\medskip	

Let us now prove that this condition is also sufficient. We consider a set of $\alpha_i$ labelled vertices with $i$ half-edges and $\pi_j$ unlabelled vertices with $j$ half-edges, with $i \geq 1$, $j \geq 2$, such that Equations \eqref{theeq} are satisfied. As Equation \eqref{eqalpi} is satisfied, we can then choose a way to associate the vertices to obtain a $\intunn$-labelled graph $T$, i.e. a graph together with a chosen bijection from $\intunn$ to a subset of its vertex set such that the image of $\intunn$ includes all of the vertices of valency $1$ and for every edge in $T$ exactly one of its endpoints lies in the image of $\intunn$.

As Equation \eqref{Eulercar} is satisfied, the Euler characteristic, i.e. the difference between the number of connected components and the number of cycles, is equal to $1$. If the graph is connected, then it has no cycles: it is a tree and we have constructed a $\intunn$-labelled tree. The associated hypertree has fixed valency and edge sizes sets.
 
If the graph is not connected, then there is a cycle in one of the connected components. Therefore, there is an edge in this connected component that can be removed without increasing the number of connected components. This edge is between an unlabelled vertex $u_1$ and a labelled vertex $l_1$. Let us cut an edge in one of the other connected components between two vertices $u_2$ and $l_2$. We then obtain a graph with each element of the set $\{u_1, l_1, u_2, l_2\}$ having an unlinked half-edge. Linking $u_2$ with $l_1$ and $u_1$ with $l_2$, we obtain a $\intunn$-labelled graph satisfying the conditions with one less connected component. Indeed, we may have disconnected the connected component of $u_2$ and $l_2$ by deleting the edge but when linking the vertices we create a path from $u_2$ to $l_2$ by using the one existing between $u_1$ and $l_1$. As this operation decreases the number of connected component, we can repeat it until we find a hypertree matching the required conditions. 

\end{proof}

\medskip

We then want to compute the coefficient $\cnap$ when it does not vanish. We do it using bijections. Given a tuple $\pi$, we call \emph{$\pi$-hooked partition} a partition with one block made of a vertex and with $\pi_j$ other blocks made of a hook and $j-1$ vertices, for all $j \geq 2$. 

\begin{exple}
A $\pi$-hooked partition $P$, for $\pi=(1,2)$:
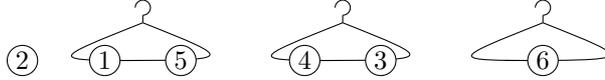
\begin{center}
\begin{tikzpicture}
\draw[black, fill=white] (0,0) circle (0. 2);
\draw(0,0) node{$2$};
\end{tikzpicture}
\begin{tikzpicture}
\draw (0.5,0) ..controls (-0.75,0) .. (0.5,0.5);
\draw (0,0) ..controls (1.75,0) .. (0.5,0.5);
\draw (0.5,0.5) -- (0.5,0.6);
\draw (0.5,0.6) arc (-90:180:0.1cm);
\draw[black, fill=white] (0,0) circle (0. 2);
\draw[black, fill=white] (1,0) circle (0. 2);
\draw(0,0) node{$1$};
\draw(1,0) node{$5$};
\end{tikzpicture}
\begin{tikzpicture}
\draw (0.5,0) ..controls (-0.75,0) .. (0.5,0.5);
\draw (0,0) ..controls (1.75,0) .. (0.5,0.5);
\draw (0.5,0.5) -- (0.5,0.6);
\draw (0.5,0.6) arc (-90:180:0.1cm);
\draw[black, fill=white] (0,0) circle (0. 2);
\draw[black, fill=white] (1,0) circle (0. 2);
\draw(0,0) node{$4$};
\draw(1,0) node{$3$};
\end{tikzpicture}
\begin{tikzpicture}
\draw (0.5,0) ..controls (-0.75,0) .. (0.5,0.5);
\draw (0,0) ..controls (1.75,0) .. (0.5,0.5);
\draw (0.5,0.5) -- (0.5,0.6);
\draw (0.5,0.6) arc (-90:180:0.1cm);
\draw[black, fill=white] (0.5,0) circle (0. 2);
\draw(0.5,0) node{$6$};
\end{tikzpicture}.
\end{center}

Then the assembly of elements of a $\pi$-hooked partition into a hypertree can be seen as an assembly of coat-hangers and coat racks. We represent here the hypertree $T$ of example \ref{exple}:
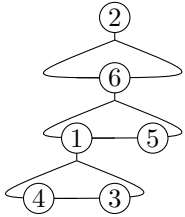
\begin{center}
\begin{tikzpicture}
\draw[black, fill=white] (0.5,0) circle (0. 2);
\draw(0.5,0) node{$2$};
\draw (0.5,-0.8) ..controls (-0.75,-0.8) .. (0.5,-0.3);
\draw (0,-0.8) ..controls (1.75,-0.8) .. (0.5,-0.3);
\draw (0.5,-0.3) -- (0.5,-0.2);
\draw[black, fill=white] (0.5,-0.8) circle (0. 2);
\draw(0.5,-0.8) node{$6$};
\draw (0.5,-1.6) ..controls (-0.75,-1.6) .. (0.5,-1.1);
\draw (0,-1.6) ..controls (1.75,-1.6) .. (0.5,-1.1);
\draw (0.5,-1.1) -- (0.5,-1);
\draw[black, fill=white] (0,-1.6) circle (0. 2);
\draw[black, fill=white] (1,-1.6) circle (0. 2);
\draw(0,-1.6) node{$1$};
\draw(1,-1.6) node{$5$};
\draw (0,-2.4) ..controls (-1.25,-2.4) .. (0,-1.9);
\draw (-0.5,-2.4) ..controls (1.25,-2.4) .. (0,-1.9);
\draw (0,-1.9) -- (0,-1.8);
\draw[black, fill=white] (-0.5,-2.4) circle (0. 2);
\draw[black, fill=white] (0.5,-2.4) circle (0. 2);
\draw(-0.5,-2.4) node{$4$};
\draw(0.5,-2.4) node{$3$};
\end{tikzpicture}.
\end{center}

For convenience purposes, we will write $X$ for the hook and represent the $\pi$-hooked partition as: 
\begin{equation*}
P= (2)\quad (X|1 \ 5)\quad (X|4 \ 3)\quad (X|6).
\end{equation*}
\end{exple}

Rooting hypertrees in one vertex, i.e. choosing one vertex in each hypertree, gives the following equation, by replacing $\cnap$ by $\frac{\cnapp}{n}$ in Equation \eqref{defCoprod}:
\begin{equation*}
\Dr (\hn)=\frac{1}{n} \sum_{(\alpha, \pi) \in \mathcal{P}_n} \cnapp p_\alpha \otimes h_\pi, 
\end{equation*}
where $\cnapp$ corresponds to the number of rooted hypertrees in $\htn$ with $\alpha_i$ vertices of valency $i$ and $\pi_j$ edges of size $j$, for all $i \geq 1$ and $j \geq 2$. 

Let us fix $\pi$ and $\alpha$ and denote by $\Pi_{\operatorname{HP}}$ the set of $\pi$-hooked partitions and by $\mathcal{H}^p_{\alpha, \pi}$, the set of rooted hypertrees with $\alpha_i$ vertices of valency $i$ and $\pi_j$ edges of size $j$. The cardinality of $\mathcal{H}^p_{\alpha, \pi}$ is $\cnapp$. We consider the map $\varphi: \mathcal{H}^p_{\alpha, \pi} \rightarrow \Pi_{\operatorname{HP}}$ defined by taking for every edge $e$ the set of all vertices of $e$, except the closest to the root, and adding a hook to this set. If we add the singleton made of the root to this set of hooked sets, we obtain a $\pi$-hooked partition. Indeed, all sets but one of cardinality one have a hook and the size of each hooked set is one less than the size of the associated edge.

Given $P$ in $\Pi_{\operatorname{HP}}$, we call $F_P$ the fibre $\varphi^{-1}(P)$. The fibres of two distinct elements of $\Pi_{\operatorname{HP}}$ are necessarily disjoint as their images by $\varphi$ are different. Moreover, any element in $\mathcal{H}^p_{\alpha, \pi}$ has an image in $\Pi_{\operatorname{HP}}$ by $\varphi$. The coefficient $\cnapp$ is then the sum of the cardinalities of the disjoint fibres. As we will see in the proof, the cardinality of a fibre is independent from the considered $\pi$-hooked partitions: we denote it by $\dnap$. We will say that we can \emph{construct} a hypertree $H$ from a $\pi$-hooked partition $P$ if $\varphi(H)=P$.

Let us now link hypertrees to hooked partitions:
\begin{lem} \label{lem1}
The coefficient $\cnap$ is linked with $\dnap$ by: 
\begin{equation}
\cnap= \frac{1}{n} \times \frac{n!}{\prod_{j \geq 2} (j-1)!^{\pi_j} \pi_j!} \times \dnap. 
\end{equation}
\end{lem}

\begin{proof} We want to compute the cardinality $\cnapp$ of $\mathcal{H}^p_{\alpha, \pi}$. Let us consider the action of the symmetric group $\mathfrak{S}_n$ on $\mathcal{H}^p_{\alpha, \pi}$. By definition of the map $\varphi$, which does not depend on the labels of the vertices, this action induces an action of the symmetric group on the set $\Pi_{\operatorname{HP}}$.
The action of the symmetric group $\mathfrak{S}_n$ on the set of all hooked partitions of type $\pi$ is transitive, as it does not change the sizes of the blocks of the partitions. We call $(\mathcal{O}_j)_{1 \leq j \leq p}$ the orbits for the action of $\mathfrak{S}_n$ on the set $\mathcal{H}^p_{\alpha, \pi}$. The fibre $F_P$ has a component $f^P_j$ in every orbit $\mathcal{O}_j$. We recap all these notations on the following diagram:

\begin{center}
\begin{figure}[h!]
\begin{tikzpicture}
\draw (10,0) rectangle (9,3);
\draw[fill=gray!40] (10,1.2) rectangle (9,1.8);
\draw[->] (8,1.5) -- (8.75,1.5) node[midway, above]{$\varphi$}; 
\draw (7.75, 0) rectangle (10-3.25,3);
\draw[fill=gray!40] (10-2.25, 1.1) rectangle (10-3.25, 1.9);
\draw (10-3.5,0) rectangle (10-5,3);
\draw[fill=gray!40] (10-3.5,1.1) rectangle (10-5,1.9);
\draw (10-5.25,0) rectangle (10-6.5,3);
\draw[fill=gray!40] (10-5.25,1.1) rectangle (10-6.5,1.9);
\draw [decorate,decoration={brace,amplitude=5pt},xshift=0pt,yshift=-1pt]
(10-0,0) -- (10-1,0) node [black,below,midway,yshift=-0.3cm] 
{\footnotesize $\Pi_{\operatorname{HP}}$};
\draw [decorate,decoration={brace,amplitude=10pt},xshift=0pt,yshift=-1pt]
(10-2.25,0) -- (10-6.5,0) node [black,midway,below, yshift=-0.3cm] 
{\footnotesize $\mathcal{H}^p_{\alpha, \pi}$};
\draw [decorate,decoration={brace,mirror, amplitude=5pt},xshift=0pt, yshift=1pt]
(10-2.25,3) -- (10-3.25,3) node [black,midway,above, yshift=0.1cm] 
{\footnotesize $\mathcal{O}_1$};
\draw [decorate,decoration={brace,mirror, amplitude=5pt},xshift=0pt, yshift=1pt]
(10-3.5,3) -- (10-5,3) node [black,midway,above, yshift=0.1cm] 
{\footnotesize $\mathcal{O}_2$};
\draw [decorate,decoration={brace,mirror, amplitude=5pt},xshift=0pt, yshift=1pt]
(10-5.25,3) -- (10-6.5,3) node [black,midway,above,yshift=0.1cm] 
{\footnotesize $\mathcal{O}_3$};
\draw (10-0.5,1.5) circle (0.3);
\draw (10-0.5,1.5) node{$P$};
\draw (10-5.87,1.5) node{$H_3$};
\draw (10-4.25,1.5) node{$H_2$};
\draw (10-2.75,1.5) node{$H_1$};
\draw [decorate,decoration={brace,mirror, amplitude=5pt},xshift=0pt, yshift=1pt]
(10-5.25,1.9) -- (10-6.5,1.9) node [black,midway,above,yshift=0.2cm] 
{\small $f^P_3$};
\draw [decorate,decoration={brace,mirror, amplitude=5pt},xshift=0pt, yshift=1pt]
(10-3.5,1.9) -- (10-5,1.9) node [black,midway,above,yshift=0.2cm] 
{\small $f^P_2$};
\draw [decorate,decoration={brace,mirror, amplitude=5pt},xshift=0pt, yshift=1pt]
(10-2.25,1.9) -- (10-3.25,1.9) node [black,midway,above,yshift=0.2cm] 
{\small $f^P_1$};
\draw[decorate,decoration={brace}] 
  ([yshift=0cm]10-6.5, 1.1) -- node[left] {\footnotesize $F_P$} ([yshift=0.1cm]10-6.5,1.8);
\end{tikzpicture}
\caption{The map $\varphi$.}
\end{figure}
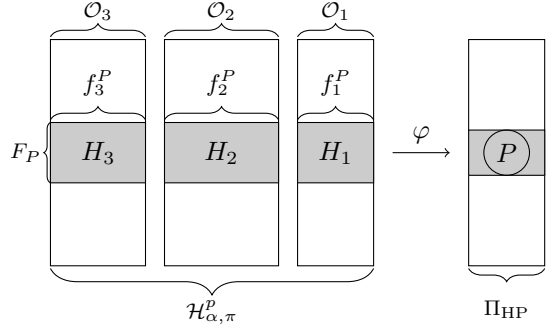
\end{center}

We consider a hypertree $H_j$ in each $f^P_j$. The orbit-stabilizer theorem applied on $\mathcal{O}_j$ gives: 
\begin{equation*}
n!=|\mathcal{O}_j| \times |\operatorname{Aut}_{H_j}|,
\end{equation*}
where $|\operatorname{Aut}_{H_j}|$ is the cardinality of the automorphism group of the rooted hypertree $H_j$.

As $\mathcal{H}^p_{\alpha, \pi} = \bigsqcup_{j=1}^p \mathcal{O}_j$, we obtain the relation: 
\begin{equation}\label{aut1}
\cnapp = n! \times \sum_{j=1}^p \frac{1}{|\operatorname{Aut}_{H_j}|}.
\end{equation}

Let us consider the group $G_P$ of permutations of $\intunn$ fixing $P$. There are exactly $\prod_{j \geq 2} (j-1)!^{\pi_j} \pi_j!$ such permutations. The group $G_P$ acts on the fibre $F_P$ transitively on each $f^P_j$. Indeed, if $\sigma \in \mathfrak{S}_n$ sends a hypertree $H$ of $f^P_k$ to a hypertree $H'$ of $f^P_k$, then as $\varphi(H)=\varphi(\sigma(H))=P$, $\sigma$ stabilizes $P$. Then, the orbit-stabilizer theorem applied on $f^P_j$ gives: 
\begin{equation}\label{eq1}
|f^P_j| \times |\operatorname{Stab}_{G_P} H_j| = |G_P| = \prod_{j \geq 2} (j-1)!^{\pi_j} \pi_j!,
\end{equation}
where $\operatorname{Stab}_{G_P} H_j = \{\sigma \in G_P| \sigma(H_j)=H_j\}$.

We show that $\operatorname{Stab}_{G_P} H_j = \operatorname{Aut}_{H_j}$. As $G_P \subseteq \mathfrak{S}_n$, it is easily shown that $\operatorname{Stab}_{G_P} H_j \subseteq \operatorname{Aut}_{H_j}$. Let us consider $\sigma$ in $\operatorname{Aut}_{H_j}$, then $\varphi(\sigma(H_j)) = \varphi(H_j) = P$ and $\varphi(\sigma(H_j)) = \sigma(P)$: $\sigma$ stabilizes $P$. Therefore, we obtain the relation $\operatorname{Stab}_{G_P} H_j = \operatorname{Aut}_{H_j}$. Combined with Equation \eqref{aut1} and Equation \eqref{eq1}, we get the result, as $\dnap = \sum_{j=1}^ p |f^P_j|$.

\end{proof}

\begin{exple} \label{exple} We consider the following $\pi$-hooked partition $P$:
\begin{equation*}
P= (2)\quad (X|1 \ 5)\quad (X|4 \  3)\quad (X|6),
\end{equation*}
with $\pi=(1, 2)$, where $X|$ represents the hook of the block. 
For $\alpha=(4,2)$, we can construct the following rooted hypertrees (and many others): 
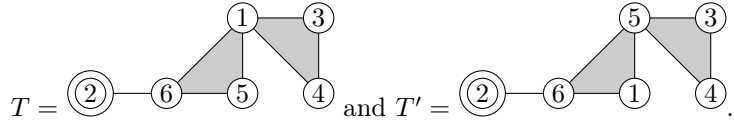
\begin{center}
$T=$
\begin{tikzpicture}
\draw[black, fill=gray!40] (0,0) -- (1,0) -- (1,1) -- (0,0);
\draw[black, fill=gray!40] (1,1) -- (2,0) -- (2,1) -- (1,1);
\draw (0,0) -- (-1,0);
\draw[black, fill=white] (0,0) circle (0. 2);
\draw[black, fill=white] (1,0) circle (0. 2);
\draw[black, fill=white] (1,1) circle (0. 2);
\draw[black, fill=white] (2,0) circle (0. 2);
\draw[black, fill=white] (2,1) circle (0. 2);
\draw[black, fill=white] (-1,0) circle (0. 3);
\draw[black, fill=white] (-1,0) circle (0. 2);
\draw(0,0) node{$6$};
\draw(1,1) node{$1$};
\draw(1,0) node{$5$};
\draw(-1,0) node{$2$};
\draw(2,0) node{$4$};
\draw(2,1) node{$3$};
\end{tikzpicture}
and $T'=$
\begin{tikzpicture}
\draw[black, fill=gray!40] (0,0) -- (1,0) -- (1,1) -- (0,0);
\draw[black, fill=gray!40] (1,1) -- (2,0) -- (2,1) -- (1,1);
\draw (0,0) -- (-1,0);
\draw[black, fill=white] (0,0) circle (0. 2);
\draw[black, fill=white] (1,0) circle (0. 2);
\draw[black, fill=white] (1,1) circle (0. 2);
\draw[black, fill=white] (2,0) circle (0. 2);
\draw[black, fill=white] (2,1) circle (0. 2);
\draw[black, fill=white] (-1,0) circle (0. 3);
\draw[black, fill=white] (-1,0) circle (0. 2);
\draw(0,0) node{$6$};
\draw(1,1) node{$5$};
\draw(1,0) node{$1$};
\draw(-1,0) node{$2$};
\draw(2,0) node{$4$};
\draw(2,1) node{$3$};
\end{tikzpicture}. 
\end{center}

We describe an example of the action of the group $G_P$ on the fibre of $P$. Considering $T$ and $T'$, which are in the fibre of $P$, the permutation $(3  \ 4)$ fixes $T$ and $T'$ but the permutation $(1 \ 5)$ sends $T$ to $T'$. Then $T$ and $T'$ are in the same orbit.

The following hypertree $T''$ is not in the orbit of $T$ and $T'$:
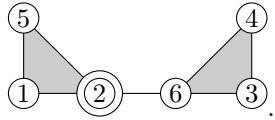
\begin{center}
\begin{tikzpicture}
\draw[black, fill=gray!40] (0,0) -- (1,0) -- (1,1) -- (0,0);
\draw[black, fill=gray!40] (-1,0) -- (-2,0) -- (-2,1) -- (-1,0);
\draw (0,0) -- (-1,0);
\draw[black, fill=white] (0,0) circle (0. 2);
\draw[black, fill=white] (1,0) circle (0. 2);
\draw[black, fill=white] (1,1) circle (0. 2);
\draw[black, fill=white] (-2,0) circle (0. 2);
\draw[black, fill=white] (-2,1) circle (0. 2);
\draw[black, fill=white] (-1,0) circle (0. 3);
\draw[black, fill=white] (-1,0) circle (0. 2);
\draw(0,0) node{$6$};
\draw(1,1) node{$4$};
\draw(1,0) node{$3$};
\draw(-1,0) node{$2$};
\draw(-2,0) node{$1$};
\draw(-2,1) node{$5$};
\end{tikzpicture}. 
\end{center}
\end{exple}
\medskip

We now want to compute the number $\dnap$ of constructions of a hypertree of valency set $\alpha$ from a $\pi$-hooked partition $P_\pi$ This is also the cardinality of the fibre $\varphi^{-1}(P_\pi)$. It is given by a bijection introduced by R. Bacher in \cite{BacHyp}, which we recall for self-containment of this article:

\begin{lem} \label{lem2}
Given a pair $(\alpha, \pi)$ in $\mathcal{P}_n$ and a $\pi$-hooked partition $P_\pi$, there is a bijection between the set of constructions of a rooted hypertree of valency set $\alpha$ from $P_\pi$ and the set of words on $\intunn$, of length $\sum_{j \geq 2} \pi_j-1$, with $\sum_{i \geq 2} \alpha_i$ different letters, where $\alpha_i$ letters appear $i-1$ times for all $i \geq 2$. 
\end{lem}

\begin{proof}
We prove this lemma using a Pr\"{u}fer code type proof. We want to count the number of different rooted hypertrees which can be constructed from a $\pi$-hooked partition $P_\pi$ and which have $\alpha_i$ vertices of valency $i$ for all $i \geq 1$. Given such a rooted hypertree, we recursively construct a variant of Pr\"{u}fer code. 

 If the hypertree has only one edge of size $n$, then we can separate the root from the edge and put a hook instead: we obtain two blocks, the one of the root and another hooked one of size $n-1$. Given a $\pi$-hooked partition, we assemble the two blocks of the partition into one edge and it gives back the hypertree. The associated word is the empty word, which is of length $0$. 

 If the rooted hypertree $H$ has more than one edge, we consider the set of leaves of the hypertree, i.e. the set of edges whose vertices but the closest from the root, called the \emph{petiole}, are of valency $1$. We can order the set of leaves according to their minimal unshared element. The petiole of the minimal leaf will be the first letter $w_1$ of the word $w$ associated with $H$. We suppose that this vertex has a valency $v$. We denote by $s_m$ the size of the minimal leaf. Then deleting the minimal leaf and its $s_m-1$ vertices different from the petiole, we obtain a rooted hypertree $H'$ on $n-s_m +1$ vertices in which the valency of the petiole $w_1$ has decreased by one, the number of vertices of valency $1$ has decreased by $s_m-1$ and all the other vertices have the same valency. As vertices of valency $1$ do not appear in the word associated with the hypertree, the deletion of these vertices only decreases by one the number of occurrences of $w_1$ in the word associated with $H'$ compared with the word associated with $H$. If $w'$ is the word associated with $H$ and $H'$, we obtain the relation $w=w_1w'$.
 
   Moreover, the hooked partition associated with $H'$ can be obtain from $P_\pi$ by deleting the hooked block of $P_\pi$ containing the vertices of valency $1$ of the minimal leaf. We then construct the word $w'$ associated with $H'$: it is a word of length $\sum_{j \geq 2} \pi_j-2$ letters, with $\sum_{i \geq 2} \alpha_i$ different letters, where $\alpha_i$ letters appear $i-1$ times for all $i \neq v, v-1$, $\alpha_v -1$ letters appear $v-1$ times and $\alpha_{v-1} +1$ letters appear $v-2$ times. Let us remark that the vertex $w_1$ is of valency $v-1$ in $H'$ so appear $v-2$ times in $w'$. Then, the letter $w_1$ appears $v-1$ times in the word $w=w_1 w'$ and the word $w=w_1 w'$ satisfies the required conditions. 

\medskip

If we have a $\pi$-hooked partition and a word $w$ satisfying the required conditions, we can build the associated rooted hypertrees by ordering the blocks with a hook whose elements are not letters of $w$ according to their minimal element. Then we attach the least element of these blocks to the last letter of the word, which is an element of another block and delete this last letter. We repeat these operations until the word is empty. We finally obtain a rooted hypertree and this operation is the inverse of the construction above. Hence, this gives a bijection between the construction of rooted hypertrees from hooked partitions and the set of words of the lemma.

\end{proof}

\begin{exple}
Considering the hooked partition $P$ and the hypertrees $T$, $T'$ and $T''$ of Example \ref{exple}, the words respectively associated to the construction of $T$, $T'$ and $T''$ from $P$ are: $1 \, 6$, $5 \, 6$, and $2 \, 6$. 

The hypertree whose construction from $P$ is associated with the word $6 \  2$ is: 
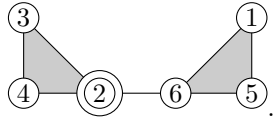
\begin{center}
\begin{tikzpicture}
\draw[black, fill=gray!40] (0,0) -- (1,0) -- (1,1) -- (0,0);
\draw[black, fill=gray!40] (-1,0) -- (-2,0) -- (-2,1) -- (-1,0);
\draw (0,0) -- (-1,0);
\draw[black, fill=white] (0,0) circle (0. 2);
\draw[black, fill=white] (1,0) circle (0. 2);
\draw[black, fill=white] (1,1) circle (0. 2);
\draw[black, fill=white] (-2,0) circle (0. 2);
\draw[black, fill=white] (-2,1) circle (0. 2);
\draw[black, fill=white] (-1,0) circle (0. 3);
\draw[black, fill=white] (-1,0) circle (0. 2);
\draw(0,0) node{$6$};
\draw(1,1) node{$1$};
\draw(1,0) node{$5$};
\draw(-1,0) node{$2$};
\draw(-2,0) node{$4$};
\draw(-2,1) node{$3$};
\end{tikzpicture}. 
\end{center}

There are $36$ words associated with the $\pi$-hooked partition: $6$ corresponding to hypertrees with a vertex of valency $3$ and the others of valency $1$, and $30$ corresponding to hypertrees with two vertices of valency $2$ and the others of valency $1$.

\end{exple}
\begin{lem} \label{lem3}
The number of words on $k$ letters, on an alphabet of size $n$, with $\alpha_i$ letters repeated $i-1$ times is: 
\begin{equation}
\dnap= \frac{k! \times n!}{\prod_{i \geq 2} (i-1)!^{\alpha_i} \alpha_i!}. 
\end{equation}
\end{lem}

\begin{proof}
The number of words on $\intunn$, of length $k$, with $\sum_{i \geq 2} \alpha_i$ different letters, where $\alpha_i$ letters appearing $i-1$ times for all $i \geq 2$ is:
\begin{equation*}
\frac{k! \times n!}{\prod_{i \geq 1} (i-1)!^{\alpha_i} \alpha_i!}. 
\end{equation*}
Indeed, there are $\binom{n}{\alpha_1, \alpha_2, \dots}$ ways to choose the letters of the word. As the letters are elements of $\intunn$, there is a natural total order on the set of letters appearing $i$ times. We consider letters according to their orders. Then, if $p$ positions in the word have already been chosen, we have $\binom{k-p}{i}$ choices for the positions of a letter appearing $i$ times. Combining these enumerations gives the result. 
\end{proof}

Thanks to this lemma, we obtain the following proposition: 

\begin{prop}  If the tuples $\alpha=(\alpha_1, \dots)$ and $\pi=(\pi_2, \dots)$ satisfy Equations \eqref{theeq}, the number of hypertrees with $\alpha_i$ vertices of valency $i$ and $\pi_j$ edges of size $j$, with $i \geq 1$ and $j \geq 2$ is given by:
\begin{equation}
\cnap= \frac{1}{n} \times \frac{n!}{\prod_{j \geq 2} (j-1)!^{\pi_j} \pi_j!} \times \frac{k! \times n!}{\prod_{i \geq 1} (i-1)!^{\alpha_i} \alpha_i!},
\end{equation}
with $k= \sum_{j\geq 2}\pi_j -1$. 
\end{prop}

\begin{proof}
This theorem follows from Lemmas \ref{lem1}, \ref{lem2} and \ref{lem3}. 
\end{proof}

This proposition associated with Equation \eqref{defCoprod} gives the coproduct: 

\begin{thm}\label{thmcoprod} If the set $\mathcal{P}(n)$ is the set of tuples $\alpha=(\alpha_1, \dots, \alpha_k)$ and $\pi=(\pi_2, \dots, \pi_l)$ satisfying Equations \eqref{crit}, the coproduct of $h_n$ in $\mathcal{B}_{HT}$ is given by:
\begin{equation*}
\Delta(h_n)= \frac{1}{n} \times \sum_{(\alpha, \pi) \in \mathcal{P}(n)} \frac{n!}{\prod_{j \geq 2} (j-1)!^{\pi_j} \pi_j!} \times \frac{k! \times n!}{\prod_{i \geq 1} (i-1)!^{\alpha_i} \alpha_i!} \prod_{i=2}^k p_i^{\alpha_i} \otimes \prod_{j=2}^l h_j^{\pi_j},
\end{equation*}
with $k= \sum_{j\geq 2}\pi_j -1$. 
\end{thm}

\begin{exple} We can now compute the coproduct of some $h_n$. Using the values of $(\alpha, \pi)$ on which $c^n_{\alpha, \pi}$ does not vanish, computed in Example \ref{criter}, we obtain for $h_4$: 
\begin{align*}
\Delta h_4 &= \frac{1}{4}\times \frac{4!}{3!} \times \frac{0!4!}{4!} \times p_1^4 \otimes h_4 +\frac{1}{4}\times \frac{4!}{2!} \times \frac{1!4!}{3!} \times p_1^3 p_2 \otimes h_2 h_3 \\&+ \frac{1}{4}\times \frac{4!}{3!} \times \frac{2!4!}{2!2!} \times p_1^2 p_2^2 \otimes h_2^3 + \frac{1}{4}\times \frac{4!}{3!} \times \frac{2!4!}{3!2!} \times p_1^3 p_3 \otimes h_2^3, \\
\Delta h_4 &= 1 \otimes h_4 + 12 \ p_2 \otimes h_2 h_3 + 12 \ p_2^2 \otimes h_2^3 + 4 \ p_3 \otimes h_2^3. 
\end{align*}
When summing the coefficients in this coproduct, we obtain as expected the total number of hypertrees on $4$ vertices, which is $29$.
\end{exple}

	\subsection{Computation of the Moebius number of the augmented hypertree posets} \label{compmoeb}

On any incidence Hopf algebra $\mathcal{H}$ with generators $g_n$ (which are posets), we define the characters $\zeta$ and $\mu$ for all $n \geq 1$ by:
\begin{equation*}
\zeta: g_n \mapsto 1
\end{equation*}
and 
\begin{equation*}
\mu: g_n \mapsto \mu(g_n),
\end{equation*}
where $\mu(g_n)$ is the Moebius number of the poset $g_n$.

These characters are the inverse of each other. It means that if $\epsilon$ is the counit of $\mathcal{H}$ and $\ast$ is the convolution on characters, we have: 
\begin{equation*}
\zeta \ast \mu = \mu \ast \zeta = \epsilon. 
\end{equation*}

Indeed, these equations come from the definitions of the convolution and the Moebius function: 
\begin{equation*}
\mu \ast \zeta ([h,h]) = \mu([h,h])= 1 
\end{equation*}
and
\begin{equation*}
\mu \ast \zeta ([h,h']) = \sum_{h \leq x \leq h'} \mu([h,x]) \times 1 = \mu(h,h') + \sum_{h \leq x < h'} \mu(h,x),
\end{equation*}
for all intervals $[h,h']$, $h<h'$ in $\mathcal{H}$. 

According to the definition of the Moebius function \ref{defmoeb}, $\mu \ast \zeta$ and $\zeta \ast \mu$ vanish on any non trivial interval. 

We want to compute the Moebius number of the augmented hypertree posets. We thus use Proposition \ref{propdec}. To prove that the characters satisfy the assumptions of the proposition, we need the following definition and lemma:
\begin{defi}
If $P$ is a finite poset with a unique least element, then, we define a \emph{sum function} by $s(P)=\sum_{x \in P} \mu (\hat{0},x)$. 
\end{defi}
If $\hat{P}$ is the poset obtained from $P$ by the addition of a greatest element $\hat{1}$, then $\mu (\hat{P}) = - s(P)$. 

\begin{lem}[Lemma 4.4 in \cite{McCM}]\label{lem4. 4} If $P_i$, $i \in [k]$ is a list of finite posets each with a unique minimal element and $Q=\prod_{i=1}^k P_i$, then $s(Q)=\prod_{i=1}^{k} s(P_i)$. 
\end{lem}

Then we can define the maps from $\Spnr$ to $\mathbb{Q}$, for 
any poset $p$ of $\Spnr$ with both a least and a greatest element and any poset $h$ with a least but no greatest element : 
\begin{equation*}
 \widetilde{\zeta} (p) =\zeta (p)=1, \text{ } \widetilde{\zeta} (h)=\zeta (h)=1,
\end{equation*} 
and
\begin{equation*}
\widetilde{\mu} (p)= \mu (p) , \text{ } \widetilde{\mu} (h) = s(h). 
\end{equation*}

These maps satisfy the following property due to their definitions and Lemma \ref{lem4. 4}, for all $i \geq 2$ and $j \geq 3$: 
\begin{equation*}
 \widetilde{\zeta} (\prod_{i=1}^{k} p_i) = \prod_{i=1}^{k} \widetilde{\zeta} (p_i), \text{ } \widetilde{\zeta} (\prod_{j=1}^{l} h_j)=\prod_{j=1}^{l} \widetilde{\zeta}(h_j),
\end{equation*} 
and
\begin{equation*}
 \widetilde{\mu} (\prod_{i=1}^{k} p_i) = \prod_{i=1}^{k} \widetilde{\mu} (p_i), \text{ } \widetilde{\mu} (\prod_{j=1}^{l} h_j)=\prod_{j=1}^{l} \widetilde{\mu}(h_j). 
\end{equation*}

As these maps satisfy the conditions of Proposition \ref{propdec}, we apply it in the following subsections. As partition and hypertree posets are not mixed in the coproduct of hypertree poset, the computation of the convolution of $\mu$ and $\zeta$ will be given by a computation using only the values of  $\widetilde{\zeta}$ and $\widetilde{\mu}$ on the partition and the hypertree posets. The first part of this section will be devoted to the equation $\zeta \ast \mu=\epsilon$ and the second part will be devoted to the equation $\mu \ast \zeta = \epsilon$. 

		\subsubsection{Right-sided computation}
		
In this section, we give a simplified proof of the result of J. McCammond and J. Meier on the computation of the Moebius number of the augmented hypertree poset. 		
		
Applying the Moebius function at the right side of the coproduct, we obtain:
\begin{equation*}
\zeta \ast \mu (\widehat{h_n})= 0,
\end{equation*}
for all $n \geq 2$. 

Hence, applying the computation of the coproduct of Theorem \ref{thmcoprod} and Proposition \ref{propdec}, we obtain the following equality for $n \geq 2$: 
\begin{equation*}
0= - \sum \widetilde{\mu}(h_n^{(2)}) + 1,
\end{equation*}
where $\Delta(h_n) = \sum h_n^{(1)} \otimes h_n^{(2)}$.

Using Lemma \ref{LMcCM}, the definition of the coproduct on $\Spnr$ and the multiplicativity of $-\mu$, we thus obtain: 
\begin{equation}\label{eqmu}
\mu(\HTn) = \sum_{\substack{ h \in \htn, \\ h>\hat{0}}} \prod_{i \in ES(h)} -\mu(\HTi) +(-1)^n,
\end{equation}
where $ES(h)$ is the multiset of sizes of the edges of $h$. 

Computing the first terms gives:
\begin{equation*}
\mu(\HTd) = -1,
\end{equation*} 
\begin{equation*}
\mu(\HTt) = 3 \times (-\mu(\HTd))^2 -1 = 2,
\end{equation*} 
and
\begin{equation*}
\mu(\HTq) = 13 \times (-\mu(\HTt))+ 16 \times (-\mu(\HTd))^3 +1 = -26+16+1 = -9. 
\end{equation*}

To obtain a closed formula, we consider the exponential generating series of hypertrees with a weight $-\mu(\HTi)$ for each edge of size $i$: 
\begin{equation*}
T(x)=-x+\sum_{n \geq 2} \sum_{ h \in \htn} \prod_{i \in ES(h)} \left(-\mu(\HTi) \right) \frac{x^n}{n!},
\end{equation*}
where $ES(h)$ is the multiset of edge sizes of hypertree $h$. 
Using Equation \eqref{eqmu}, we obtain:
\begin{equation*}
T(x)=-x-\sum_{n \geq 2} \frac{(-x)^n}{n!} = 1-e^{-x}. 
\end{equation*}

Moreover, it has been proven by Kalikow in \cite{Kalikow1999} that the derivative of $T$ satisfies the following functional equation: 
\begin{thm}[Kalikow] The generating series $T$ satisfies the following equation: 
\begin{equation*}
x T'(x)=x \times \exp(y(x)) \text{ where } y(x)=\sum_{j \geq 1} -\mu(\HTjpu) \frac{x^j T'(x)^j}{j!}
\end{equation*}
\end{thm}

We hence obtain:
\begin{equation*}
x= \sum_{j \geq 1} \mu(\HTjpu) \frac{x^j e^{-jx}}{j!}. 
\end{equation*}

This proves the following theorem by J. McCammond and J. Meier: 
\begin{thm}[Theorem 5.1 in \cite{McCM}]\label{McCM}
The Moebius number of the augmented hypertree poset on $n$ vertices is given by: 
\begin{equation*}
\mu(\HTn) = (-1)^{n-1} (n-1)^{n-2}. 
\end{equation*}
\end{thm}

As the homology of the augmented hypertree poset is concentrated in top degree, this Moebius number is also the dimension of the only homology group of the hypertree poset. The action of the symmetric group on this homology group has been computed in \cite{mar2}. 

		\subsubsection{Left-sided computation}
		
Applying the Moebius function at the left side of the coproduct, we obtain:
\begin{equation*}
\mu \ast \zeta (\widehat{h_n})= 0,
\end{equation*}
for all $n \geq 2$. 

By Proposition \ref{propdec}, this can be rewritten for all $n \geq 2$ as: 
\begin{equation*}
0 = \sum \widetilde{\mu}(h_n^{(1)}) \widetilde{\zeta}(h_n^{(2)}) - \widetilde{\mu}(h_n). 
\end{equation*}

The formula \eqref{defCoprod} for the coproduct gives:
\begin{equation*}
\mu(\HTn) = -\sum_{(\alpha, \pi)\in \mathcal{P}_n} c^n_{\alpha, \pi} \prod (-1)^{(i-1)\alpha_i} (i-1)!^{\alpha_i} 
\end{equation*}

Using Theorem \ref{McCM} and Theorem \ref{thmcoprod}, we obtain the following proposition: 
\begin{prop}The following equality holds:
\begin{equation*}
(n-1)^{n-2} = \sum_{(\alpha, \pi) \in \mathcal{P}(n)} \frac{(-1)^{i \alpha_i-1}}{n} \times \frac{n!}{\prod_{j \geq 2} (j-1)!^{\pi_j} \pi_j!} \times \frac{k! \times n!}{\prod_{i \geq 1} \alpha_i!},
\end{equation*}
where $\mathcal{P}(n)$ is the set of pairs of tuples $\left( \alpha = (\alpha_1, \dots, \alpha_k),\pi=(\pi_2, \dots, \pi_l) \right)$ satisfying:
\begin{equation*}
\sum_{i=1}^k \alpha_i = n, \quad \sum_{j=2}^{l} (j-1) \pi_j = n-1, \quad \text{and} \quad \sum_{i=1}^k i \alpha_i = n + \sum_{j=2}^{l} \pi_j -1. 
\end{equation*} 
\end{prop}

\begin{proof}
This comes from the computation of the coproduct, combined with the Moebius numbers of the augmented hypertree posets and of the partition posets. Indeed, the Moebius number of the partition poset on $n$ elements is given by $(-1)^{n-1} (n-1)!$. 
\end{proof}

\begin{exple}
The first terms obtained are: 
\begin{equation*}
\mu(\HTq)=-1+12-12-8=-9
\end{equation*}
and
\begin{equation*}
\mu(\HTc)=-1+20+12-120-60+60+120+30 = 64. 
\end{equation*}

\end{exple}

\bibliographystyle{alpha}
\bibliography{bibli}

\begin{thebibliography}{BMC12}

\bibitem[Bac11]{BacHyp}
Roland Bacher.
\newblock On the enumeration of labelled hypertrees and of labelled bipartite
  trees.
\newblock arXiv:1102.2708, 2011.

\bibitem[Ber89]{Berge}
Claude Berge.
\newblock {\em Hypergraphs}, volume~45 of {\em North-Holland Mathematical
  Library}.
\newblock North-Holland Publishing Co., Amsterdam, 1989.
\newblock Combinatorics of finite sets.

\bibitem[BMC12]{BMC}
Mireille Bousquet-M{\'e}lou and Guillaume Chapuy.
\newblock The vertical profile of embedded trees.
\newblock {\em Electron. J. Combin.}, 19(3):Paper 46, 61, 2012.

\bibitem[Cha07]{ChHyp}
Fr\'ed\'eric Chapoton.
\newblock Hyperarbres, arbres enracin\'es et partitions point\'ees.
\newblock {\em Homology, Homotopy Appl.}, 9(1):193--212, 2007.
\newblock \url{http://www.intlpress.com/hha/v9/n1/}.

\bibitem[Kal99]{Kalikow1999}
Louis~H. Kalikow.
\newblock {\em Enumeration of parking functions, allowable permutation pairs,
  and labeled trees}.
\newblock ProQuest LLC, Ann Arbor, MI, 1999.
\newblock Thesis (Ph.D.)--Brandeis University.

\bibitem[MM96]{McCuMi}
Darryl McCullough and Andy Miller.
\newblock Symmetric automorphisms of free products.
\newblock {\em Mem. Amer. Math. Soc.}, 122(582):viii+97, 1996.

\bibitem[MM04]{McCM}
Jon McCammond and John Meier.
\newblock The hypertree poset and the {$l^2$}-{B}etti numbers of the motion
  group of the trivial link.
\newblock {\em Math. Ann.}, 328(4):633--652, 2004.

\bibitem[Oge13a]{mar1}
B{\'e}r{\'e}nice Oger.
\newblock Action of the symmetric groups on the homology of the hypertree
  posets.
\newblock {\em J. Algebraic Combin.}, 38(4):915--945, 2013.

\bibitem[Oge13b]{mar2}
B{\'e}r{\'e}nice Oger.
\newblock Decorated hypertrees.
\newblock {\em J. Combin. Theory Ser. A}, 120(7):1871--1905, 2013.

\bibitem[Sch87]{AIC}
William~R. Schmitt.
\newblock Antipodes and incidence coalgebras.
\newblock {\em J. Combin. Theory Ser. A}, 46(2):264--290, 1987.

\bibitem[Sch94]{IHA}
William~R. Schmitt.
\newblock Incidence {H}opf algebras.
\newblock {\em J. Pure Appl. Algebra}, 96(3):299--330, 1994.

\bibitem[Spe97]{Speicher1997}
Roland Speicher.
\newblock Free probability theory and non-crossing partitions.
\newblock {\em S\'em. Lothar. Combin.}, 39:Art.\ B39c, 38 pp.\ (electronic),
  1997.

\bibitem[Sta01]{Stanley2001}
Richard~P. Stanley.
\newblock {\em Enumerative Combinatorics}.
\newblock Number vol.~2 in Cambridge Studies in Advanced Mathematics. Cambridge
  University Press, 2001.

\end{thebibliography}

\end{document}